\documentclass[12pt]{elsarticle}
\usepackage{amsfonts,amssymb,mathrsfs, amsthm, textcomp,amscd,amsbsy,amsmath}
\usepackage[small,nohug,heads=littlevee]{diagrams}
\diagramstyle[labelstyle=\scriptstyle]

\usepackage{verbatim}
\usepackage{enumerate}
\usepackage{amsmath}
\usepackage{geometry}
\usepackage[latin5]{inputenc}
\usepackage{yfonts}

   \newtheorem{theorem}{Theorem}[section]
   \newtheorem{lemma}[theorem]{Lemma}

\newtheorem{remark}[theorem]{Remark}
\newtheorem{prop}[theorem]{Proposition}
\newtheorem{corollary}[theorem]{Corollary}
\newtheorem{definition}[theorem]{Definition}

\newcommand{\Aut}{\mathrm{Aut}}

\newcommand{\Def}{\mathrm{Def}}

\newcommand{\FF}{\mathbb{F}}
\newcommand{\FFT}{\mathbb{F}T}
\newcommand{\FFTD}{\mathbb{F}T^{\Delta}}

\newcommand{\Hom}{\mathrm{Hom}}

\newcommand{\Ind}{\mathrm{Ind}}

\newcommand{\Inf}{\mathrm{Inf}}

\newcommand{\Isom}{\mathrm{Iso}}

\newcommand{\NN}{\mathbb{N}}

\newcommand{\Res}{\mathrm{Res}}
\newcommand{\Out}{\mathrm{Out}}

\newcommand{\ZZ}{\mathbb{Z}}
\newcommand{\DD}{D^{\Delta}}
\newcommand{\TD}{T^{\Delta}}
\newcommand{\beps}{\mathbf{e}_{P,s}}
\newcommand{\beqt}{\mathbf{e}_{Q,t}}

\newcommand{\Ps}{\langle Ps \rangle}
\newcommand{\Qt}{\langle Qt \rangle}
\newcommand{\Indinf}{\mathrm{Indinf}}
\def\sur{\overline}

\usepackage{amssymb}

\journal{Journal of Algebra}

\begin{document}

\begin{frontmatter}



\title{Diagonal $p$-permutation functors}


\author{Serge Bouc, Deniz Y{\i}lmaz\corref{cor1}\fnref{fn2}}

\address{CNRS-LAMFA Universit{\'e} de Picardie-Jules Verne, 33, rue St Leu, 80039, Amiens Cedex 01 - France} 

\address{University of California, Santa Cruz Department of Mathematics CA 95064 USA}

\cortext[cor1]{The second author is supported by the Chateaubriand Fellowship of the Office of Science and Technology of the Embassy of France in the United States.}
\fntext[fn2]{The second author is thankful to LAMFA for their hospitality during the visit.}
\begin{footnotesize}
\begin{abstract}
Let $k$ be an algebraically closed field of positive characteristic $p$, and $\FF$ be an algebraically closed field of characteristic 0. We consider the $\FF$-linear category $\FF pp_k^\Delta$ of finite groups, in which the set of morphisms from $G$ to $H$ is the $\FF$-linear extension $\FF T^\Delta(H,G)$ of the Grothendieck group $T^\Delta(H,G)$ of $p$-permutation $(kH,kG)$-bimodules with (twisted) diagonal vertices. The $\FF$-linear functors from $\FF pp_k^\Delta$ to $\FF\hbox{\rm-Mod}$ are called {\em diagonal $p$-permutation functors}. They form an abelian category $\mathcal{F}_{pp_k}^\Delta$.\par
We study in particular the functor $\FFTD$ sending a finite group $G$ to the Grothendieck group $\FFT(G)$ of $p$-permutation $kG$-modules, and show that $\FFTD$ is a semisimple object of $\mathcal{F}_{pp_k}^\Delta$, equal to the direct sum of specific simple functors parametrized by isomorphism classes of pairs $(P,s)$ of a finite $p$-group $P$ and a generator $s$ of a $p'$-subgroup acting faithfully on $P$. This leads to a precise description of the evaluations of these simple functors. In particular, we show that the simple functor indexed by the trivial pair $(1,1)$ is isomorphic to the functor sending a finite group $G$ to $\FF K_0(kG)$, where $K_0(kG)$ is the group of projective $kG$-modules.

\end{abstract}



\end{footnotesize}

\end{frontmatter}


\section{Introduction}
Let $p$ be a prime number. Throughout we denote by $\FF$ an algebraically closed field of characteristic zero, and by $k$ an algebraically closed field of characteristic $p$. The $p$-permutation modules play a crucial role in the study of modular representation theory of finite groups. A splendid Rickard equivalence, introduced by Rickard \cite{Rickard}, between blocks of finite group algebras is given by a chain complex consisting of $p$-permutation bimodules. Also a $p$-permutation equivalence, introduced by Boltje and Xu \cite{BoltjeXu}, and studied extensively later by Boltje and Perepelitsky \cite{Philipp}, is an element in the Grothendieck group of $p$-permutation bimodules. 

In \cite{MD}, Ducellier studied $p$-permutation functors: Consider the category $\FF pp_k$ where the  objects are finite groups and the morphisms between groups $G$ and $H$ are given by the Grothendieck group $\FF\otimes_{\ZZ}T(H,G)$ of $p$-permutation $(kH,kG)$-bimodules. A $p$-permutation functor is an $\FF$-linear functor from $\FF pp_k$ to $\FF$-$\mathrm{Mod}$. The indecomposable direct summands of the bimodules that appears in a $p$-permutation equivalence between blocks of finite group algebras have twisted diagonal vertices. Therefore, inspired by the work of Ducellier, we consider a category with less morphisms: Let $\FF pp_k^\Delta$ be a category where the objects are finite groups and the morphisms between groups $G$ and $H$ are given by the Grothendieck group $\FF\otimes_{\ZZ}T^{\Delta}(H,G)$ of $p$-permutation $(kH,kG)$-bimodules whose indecomposable direct summands have twisted diagonal vertices. An $\FF$-linear functor from $\FF pp_k^\Delta$ to $\FF$-$\mathrm{Mod}$ is called a \textit{diagonal $p$-permutation functor}. 

By \cite{Boucsimple}, the simple diagonal $p$-permutation functors are parametrized by the pairs $(G,V)$ of a finite group $G$ and a simple module $V$ of the essential algebra $\mathcal{E}^{\Delta}(G)=\mathrm{End}_{\FF pp_k^\Delta}(G)/I$ at $G$, where $I$ is the ideal generated by the morphisms that factor through groups of smaller order. We show that the essential algebra  $\mathcal{E}^{\Delta}(G)$ is isomorphic to the essential algebra studied in \cite{MD}. As a result this implies that the essential algebra  $\mathcal{E}^{\Delta}(G)$ is non-zero if and only if the group $G$ is of the form $P\rtimes \langle s\rangle$ where $P$ is a $p$-group and $s$ is a generator of a $p'$-cyclic group acting faithfully on $P$. Moreover in that case there is an algebra isomorphism  $\mathcal{E}^{\Delta}(G)\cong \big(\FF[X]/\Phi_n[X]\big) \rtimes \Out(G)$ where $n$ is the order of $s$. See Theorem \ref{essentialfinal}.  

We also study the functor $\FFTD$ that sends a finite group~$G$ to the Grothendieck group $\FF T(G)$ of $p$-permutation $kG$-modules. We describe the subfunctor lattice (Theorem \ref{subfunctorlattice}) and simple quotients (Proposition \ref{simplequotient}) of $\FFTD$. We also give a description for the $\FF$-dimension of the evaluations of simple quotients of $\FFTD$ at a finite group $G$ (Theorem \ref{dimension}). Moreover we prove that the simple functor $S_{1,1}$ that corresponds to the pair $(1,1)$ is isomorphic to the functor that sends a finite group~$G$ to the $\FF$-linear extension $\FF K_0(kG)$ of the Grothendieck group of projective $kG$-modules (Theorem \ref{projectivefunctor}).

\section{Preliminaries}
Let $G$ and $H$ be finite groups. We denote by $p_1: G\times H\to G$ and $p_2:G\times ~H\to H$ the canonical projections. Let $X\leqslant G\times H$ be a subgroup. We define the subgroups $k_1(X):=p_1(X\cap \mathrm{ker}(p_2))$ and $k_2(X):=p_2(X\cap \mathrm{ker}(p_1))$ of $p_1(X)$ and $p_2(X)$, respectively. Note that $k_1(X)\times k_2(X)$ is a normal subgroup of $X$. Moreover, $k_i(X)$ is a normal subgroup of $p_i(X)$ and one has a canonical isomorphism $X/(k_1(X)\times k_2(X))\to p_i(X)/k_i(X)$ induced by the projection map $p_i$ for $i=1,2$. 

Let $\phi:P\to Q$ be an isomorphism between subgroups $P\leqslant G$ and $Q\leqslant H$. Then $\{(\phi(x),x): x\in P\}$ is a subgroup of $H\times G$ and a subgroup of that form is called a \textit{twisted diagonal} subgroup of $H\times G$. Note that a subgroup $X\leqslant H\times G$ is a twisted diagonal subgroup if and only if $k_1(X)=1$ and $k_2(X)=1$.   

Let $P$ be a subgroup of $G$ and $M$ be a $kG$-module. We denote by $M^P$ the $k$-vector space of $P$-fixed points of $M$. If $Q\leqslant P$ is a subgroup, then the map $\mathrm{Tr}_Q^P: M^Q\to M^P$ defined by $\mathrm{Tr}(m)= \sum_{x\in [P/Q]} x\cdot m$ is called the \textit{relative trace map}.  The quotient 
\[
M[P]:=M^P/\sum_{Q< P} \mathrm{Tr}_Q^P(M^Q)
\]
is called the \textit{Brauer quotient} of $M$ at $P$. Note that $M[P]$ is a $k\overline{N}_G(P)$-module, where $\overline{N}_G(P):=N_G(P)/P$. We have $M[P]=0$ if $P$ is not a $p$-group.  

A $(kG,kH)$-bimodule $M$ can be viewed as a $k(G\times H)$-module via $(g,h)\cdot m:= gmh^{-1}$, for $(g,h)\in G\times H$ and $m\in M$. Similarly a $k(G\times H)$-module can be viewed as a $(kG,kH)$-bimodule. We will usually switch between these two points of views. 

\begin{definition}
Let $G$ be a finite group. A $kG$-module $M$ is called a \textit{permutation} module, if $M$ has a $G$-stable $k$-basis. A $p$-permutation $kG$-module is a $kG$-module $M$ such that $\Res^G_S M$ is a permutation $kS$-module for a Sylow $p$-subgroup $S$ of $G$. 
\end{definition}

For a finite group $G$ we denote by $T(G)$ the Grothendieck group of $p$-permutation $kG$-modules with respect to direct sum decompositions. If $M$ is a $p$-permutation $kG$-module, then the class of $M$ in $T(G)$ will be abusively denoted by $M$. The group $T(G)$ has a ring structure induced by the tensor product of modules over $k$, and $T(G)$ will be called the {\em ring of $p$-permutation modules} of $G$, for short. If $H$ is another finite group, we set $T(G,H):=T(G\times H)$.  We denote by $T^{\Delta}(G,H)$ the subgroup of $T(G,H)$ spanned by $p$-permutation $k(G\times H)$-modules whose indecomposable direct summands have twisted diagonal vertices. 

Let $\mathcal{P}_{G,p}$ denote the set of pairs $(P,E)$ where $P$ is a $p$-subgroup of $G$ and $E$ is a projective indecomposable $k\overline{N}_G(P)$-module. The group $G$ acts on the set $\mathcal{P}_{G,p}$ via conjugation and we denote by $[\mathcal{P}_{G,p}]$ a set of representatives of $G$-orbits of $\mathcal{P}_{G,p}$. For $(P,E)\in \mathcal{P}_{G,p}$, let $M_{P,E}$ denote the unique (up to isomorphism) indecomposable $p$-permutation $kG$-module with the property that $M_{P,E}[P]\cong E$. Note that $M_{P,E}$ has the group $P$ as a vertex \cite[Theorem 3.2]{Broue}. We denote by $\mathcal{P}_{G\times H,p}^{\Delta}$ the set of pairs $(P,E)\in \mathcal{P}_{G\times H,p}$ where $P$ is a twisted diagonal $p$-subgroup of $G\times H$. 

\begin{remark}
The isomorphism classes of the modules $M_{P,E}$ where $(P,E)\in \mathcal{P}_{G\times H,p}^{\Delta}$ form a $\ZZ$-basis for $T^{\Delta}(G,H)$.
\end{remark}

\begin{definition}\cite[Definition 2.3.1]{MD}
Let $(P,s)$ be a pair where $P$ is a $p$-group and $s$ is a generator of a $p'$-cyclic group acting on $P$. We denote the semidirect product $P\rtimes \langle s\rangle$ by $\Ps$. Let $(Q,t)$ be another such pair. We say that the pairs $(P,s)$ and $(Q,t)$ are isomorphic if there are group isomorphisms $\phi: P\to Q$ and $\psi:\langle s\rangle\to \langle t\rangle$ such that $\psi(s)=q\cdot t$ for some $q\in Q$ and $\phi(s\cdot u)=\psi(s)\cdot\phi (u)$ for all $u\in P$. In that case we write $(P,s)\simeq (Q,t)$. 
\end{definition}

\begin{lemma}\cite[Proposition 2.3.3]{MD}
Let $(P,s)$ and $(Q,t)$ be two pairs. Then $(P,s)\simeq (Q,t)$ if and only if there is a group isomorphism $f:\Ps\to \Qt$ such that $f(s)$ is conjugate to $t$.
\end{lemma}

Let $\mathcal{Q}_{G,p}$ denote the set of pairs $(P,s)$ where $P$ is a $p$-subgroup of $G$ and $s\in  N_{G}(P)$ is a $p'$-element. In that case $\Ps$ denotes the semidirect product $P\rtimes \langle s\rangle$ where the action of $\langle s\rangle$ on $P$ is induced by conjugation. The group $G$ acts on the set $\mathcal{Q}_{G,p}$ and we denote by $[\mathcal{Q}_{G,p}]$ a set of representatives of $G$-orbits. We denote by $\mathcal{Q}_{G\times H,p}^{\Delta}$ the set of pairs $(P,s)\in \mathcal{Q}_{G\times H,p}$ where $P$ is a twisted diagonal $p$-subgroup of $G\times H$.

For any pair $(P,s)\in \mathcal{Q}_{G,p}$ let $\tau^G_{P,s}$ denote the additive map $T(G)\to \FF$ that sends a $p$-permutation $kG$-module $M$ to the value of the Brauer character of $M[P]$ at $s$. The map $\tau^G_{P,s}$ is a ring homomorphism and it extends to an $\FF$-algebra homomorphism $\tau^G_{P,s}: \FF\otimes_{\ZZ} T(G)\to \FF$. The set $\{\tau^G_{P,s}: (P,s)\in [\mathcal{Q}_{G,p}]\}$ is the set of all species from $\FFT(G):=\FF\otimes_{\ZZ} T(G)$ to $\FF$ \cite[Proposition 2.18]{BT}. 

The algebra $\FFT(G)$ is split semisimple and its primitive idempotents $F^{G}_{P,s}$ are indexed by pairs $(P,s)\in [\mathcal{Q}_{G,p}]$ \cite[Corollary 2.19]{BT}. If $\phi:\langle s\rangle\to k^{\times}$ is a group homomorphism, we denote by $k_{\phi}$ the $k\langle s\rangle$-module $k$ on which the element $s$ acts as multiplication by $\phi(s)$.  Let $\widehat{\langle s\rangle}=\Hom(\langle s\rangle, k^{\times})$ denote the set of group homomorphisms.   By \cite[Theorem 4.12]{BT} we have the idempotent formula 
\[
F^G_{P,s} = \frac{1}{|P||s||C_{\overline{N}_G(P)}(s)|}\sum_{\substack{\varphi\in \widehat{\langle s\rangle} \\ L\leqslant \Ps \\ PL=\Ps}} \tilde{\varphi}(s^{-1})|L|\mu(L,\Ps)\Ind^G_L k^{\Ps}_{L,\varphi},
\]
where $k^{\Ps}_{L,\varphi}=\Res^{\Ps}_L \Inf^{\Ps}_{\langle s\rangle} k_{\varphi}$, and $\tilde{\varphi}$ is the Brauer character of $k_{\varphi}$.

By \cite[Proposition 2.7.8]{MD} we have another formula 
\[
F^G_{P,s} = \frac{1}{|C_{N_G(P)}(s)|}\sum_{\substack{\varphi\in \widehat{\langle s\rangle} \\ L\leqslant P \\ L^s=L}} \tilde{\varphi}(s^{-1})|C_L(s)|\mu\big((L,P)^s\big)\Ind^G_{\langle Ls\rangle} k^{\Ps}_{\langle Ls\rangle,\varphi}.
\]
Here $\mu\big((-,-)^s\big)$ is the M{\"o}bius function of the poset of $s$-stable subgroups of $P$.

\begin{lemma}
For finite groups $G$ and $H$, the set $\{F^{G\times H}_{P,s} : (P,s) \in [\mathcal{Q}_{G\times H,p}^{\Delta}]\}$ of primitive idempotents form an $\FF$-basis for the split semisimple algebra $\FFT^{\Delta}(G,H)$. 
\end{lemma}
\begin{proof}
First we will show that we have $F^{G\times H}_{P,s}\in \FFTD(G,H)$ whenever $(P,s)\in [\mathcal{Q}_{G\times H,p}^{\Delta}]$. Let $\varphi\in \widehat{\langle s\rangle}$ and $L\leqslant \Ps$. It suffices to show that $\Ind^G_L k^{\Ps}_{L,\varphi}\in \FFTD(G,H)$. Since $P$ acts trivially on $\Inf^{\Ps}_{\langle s\rangle} k_{\varphi}$, the subgroup $P$ is contained in a vertex of $k_{\varphi}$ considered as a $k\Ps$-module. But since $P$ is the Sylow $p$-subgroup of $\Ps$, it follows that $P$ is the vertex of $k_{\varphi}$. Therefore the module $k^{\Ps}_{L,\varphi}=\Res^{\Ps}_L \Inf^{\Ps}_{\langle s\rangle} k_{\varphi}$ has a vertex contained in $L\cap {}^{x}P\leqslant P$ for some $x\in \Ps$. Since a subgroup of twisted diagonal subgroup is again twisted diagonal, this means that $k^{\Ps}_{L,\varphi}$ has twisted diagonal vertices. This shows that $\Ind^G_L k^{\Ps}_{L,\varphi}\in \FFTD(G,H)$ as desired. Now since the $\FF$-dimension of $\FFTD(G,H)$ is equal to the cardinality of $[\mathcal{P}_{G\times H,p}^{\Delta}]$, which is equal to the cardinality of $ [\mathcal{Q}_{G\times H,p}^{\Delta}]$, it follows that the set $\{F^{G\times H}_{P,s} : (P,s) \in [\mathcal{Q}_{G\times H,p}^{\Delta}]\}$ of primitive idempotents form an $\FF$-basis for $\FFTD(G,H)$. 
\end{proof}

Let $G,H$ and $L$ be finite groups. If $X$ is a $(kG,kH)$-bimodule and $Y$ is a $(kH,kL)$-bimodule, then $X\circ Y:=X\otimes_{kH}Y$ is a $(kG,kL)$-bimodule. Extending this product by $\FF$-bilinearity, we get a map 
\[
\FFT(G,H)\circ \FFT(H,L)\to \FFT(G,L).
\]
Note that this induces a map 
\[
\FFTD(G,H)\circ \FFTD(H,L)\to \FFTD(G,L)
\]
which is used to define the composition of morphisms in the following category.

\begin{definition}
Let $\FF pp_k^\Delta$ be the category with
\begin{itemize}
\item objects: finite groups
\item $\mathrm{Mor}_{\FF pp_k^\Delta}(G,H) = \FF\otimes_{\ZZ}T^{\Delta}(H,G)=\FFT^{\Delta}(H,G)$.
\end{itemize} 
\end{definition}
An $\FF$-linear functor from $\FF pp_k^\Delta$ to $\FF$-$\mathrm{Mod}$ is called a \textit{diagonal $p$-permutation functor}. Diagonal $p$-permutation functors form an abelian category $\mathcal{F}_{pp_k}^\Delta$.
\section{The Essential Algebra}
For a finite group $G$, the quotient algebra
\begin{align*}
\mathcal{E}^{\Delta}(G):=\FFT^{\Delta}(G,G)/ \Big(\sum_{|H|<|G|} \FFT^{\Delta}(G,H)\circ \FFT^{\Delta}(H,G)\Big)
\end{align*}
is called the \textit{essential algebra} of $G$. 

By \cite[Proposition 4.1.2 and Theorem 4.1.12]{MD} the algebra 
\[
\mathcal{E}(G):=\FFT(G,G)/ \Big(\sum_{|H|<|G|} \FFT(G,H)\circ \FFT(H,G)\Big)
\]
is non-zero if and only if there exists a pair $(P,s)$ in $G$ such that $G=\Ps$ and $C_{\langle s\rangle}(P)=1$. In that case, we also have an algebra isomorphism
\[
\mathcal{E}(G)\cong \big(\FF[X]/\Phi_n[X]\big) \rtimes \Out(G)
\]
 where $n$ is the order of $s$ \cite[Theorem 4.1.12]{MD}.

Note that the inclusion map $\FFTD(G,G)\hookrightarrow \FFT(G,G)$ induces a map 
\[
\Theta:\mathcal{E}^{\Delta}(G)\to \mathcal{E}(G).
\] 
We will show that this map is an algebra isomorphism. 

Let $\varphi\in \Aut(G)$ be an automorphism and $\lambda: G/O_p(G)\to k^\times$ be a character, where $O_p(G)$ denotes the largest normal $p$-subgroup of $G$. We define a $(kG, kG)$-bimodule structure on $kG$, denoted by $kG_{\varphi,\lambda}$, via
\[
a\cdot g\cdot b:= \lambda(b)ag\varphi(b)
\]
for $a,b,g\in G$.

Let $\langle Rt\rangle$ be a twisted diagonal subgroup of $G\times G$ with $p_1(\langle Rt\rangle)=G$ and $p_2(\langle Rt\rangle)=G$. Let also $\eta:p_1(\langle Rt\rangle)\to p_2(\langle Rt\rangle)$ be the canonical isomorphism. Then by \cite[Section 4.1.2]{MD}  we have an isomorphism
\[
\Ind^{G\times G}_{\langle Rt\rangle} k^{\langle Rt\rangle}_{\langle Rt\rangle, \varphi}\cong kG_{\eta^{-1},\varphi^{-1}}
\]
 of $(kG,kG)$-bimodules. Again by \cite[Section 4.1.2]{MD} the algebra $\mathcal{E}(G)$ is generated by the images of $kG_{\varphi,\lambda}$.

\begin{prop}
If the essential algebra $\mathcal{E}^{\Delta}(G)$ of a finite group $G$ is non-zero, then there exists a pair $(P,s)$ in $G$ such that $G=\Ps$ and $C_{\langle s\rangle}(P)=1$.
\end{prop}
\begin{proof}
Let $(Q,t)$ be a pair contained in $G\times G$ such that $Q$ is a twisted diagonal subgroup and recall the idempotent formula
\[
F^{G\times G}_{Q,t} = \frac{1}{|C_{N_{G\times G(Q)}(t)}|}\sum_{\substack{\varphi\in \widehat{\langle t\rangle} \\ L\leqslant Q \\ L^t=L}} \tilde{\varphi}(t^{-1})|C_L(t)|\mu((L,Q)^t)\Ind^{G\times G}_{\langle Lt\rangle} k^{\Qt}_{L,\varphi}.
\]
By \cite[Lemma 2.5.9]{MD} we have an isomorphism 
\begin{align*}
\Ind^{G\times G}_{\langle Lt\rangle} k^{\Qt}_{L,\varphi}  & \cong \Ind^G_{p_1(\langle Lt\rangle)} \otimes_{p_1(\langle Lt\rangle)} \Ind^{p_1(\langle Lt\rangle)\times p_2(\langle Lt\rangle)}_{\langle Lt
\rangle} (k^{\Qt}_{L,\varphi})\otimes_{p_2(\langle Lt\rangle)}\Res^G_{p_2(\langle Lt\rangle)}\\&
\cong kG \otimes_{p_1(\langle Lt\rangle)} \Ind^{p_1(\langle Lt\rangle)\times p_2(\langle Lt\rangle)}_{\langle Lt
\rangle} (k^{\Qt}_{L,\varphi})\otimes_{p_2(\langle Lt\rangle)} kG
\end{align*}
of $(kG,kG)$-bimodules. As $(kG, kG)$-bimodule, we have the isomorphism $kG\cong \Ind^{G\times G}_{\Delta G} k$. Thus as $(kG,kp_1(\langle Lt\rangle))$-bimodule we have,
\[
\Res^{G\times G}_{G\times p_1(\langle Lt\rangle)} kG\cong \Res^{G\times G}_{G\times p_1(\langle Lt\rangle)} \Ind^{G\times G}_{\Delta G} k\cong \Ind^{G\times p_1(\langle Lt\rangle)}_{\Delta(p_1(\langle Lt\rangle))}\Res^{\Delta(G)}_{\Delta(p_1(\langle Lt\rangle))} k.
\]
Therefore as $(kG\times p_1(\langle Lt\rangle))$-module, the indecomposable direct summands of $kG$ have vertices contained in $\Delta(p_1(\langle Lt\rangle))$. Similary, one can show that the indecomposable direct summands of $kG$ as $k(p_2(\langle Lt\rangle)\times G)$-module, have vertices contained in $\Delta(p_2(\langle Lt\rangle))$. We also know that the module $k^{\Qt}_{L,\varphi}$, and hence the indecomposable direct summands of $\Ind^{p_1(\langle Lt\rangle)\times p_2(\langle Lt\rangle)}_{\langle Lt\rangle} (k^{\Qt}_{L,\varphi})$, have twisted diagonal vertices. Now suppose $\mathcal{E}^{\Delta}(G)$ is non-zero. Then there is an idempotent $F^{G\times G}_{Q,t}$ whose image in $\mathcal{E}^{\Delta}(G)$ is non-zero. Therefore the argument above shows that there is a pair $(Q,t)$ in $G\times G$ such that $p_1(\Qt)=G$ and $p_2(\Qt)=G$. This implies that  there is a $p$-subgroup $P$ of $G$ and a $p'$-element $s$ of $G$ that normalises $P$ such that $G=\Ps$. Now we will show that in that case we have $C_{\langle s\rangle}(P)=1$.\\
Let $\sur{G}:=G/C_{\langle s\rangle}(P)$, $Q:=\{(u,\sur{u}: u\in P\}\leqslant G\times \sur{G}$ and $Q':=\{(\sur{u},u):u\in P\}\leqslant \sur{G}\times G$. Then by \cite[Proof of Proposition 4.1.2]{MD} we have an isomorphism between $kG$ and
\[
\bigoplus_{i} \Indinf^{G\times \sur{G}}_{\sur{N}_{G\times \sur{G}}(Q)}\big(k C_G(P)/C_{\langle s\rangle}(P)\otimes_k k_{\alpha_i}\big)\otimes_{k\sur{G}} \Indinf^{\sur{G}\times G}_{\sur{N}_{\sur{G}\times G}(Q')}\big(k C_G(P)/C_{\langle s\rangle}(P)\otimes_k k_{\alpha'_i}\big)
\]
as $(kG,kG)$-bimodules, where $\Indinf^{G\times \sur{G}}_{\sur{N}_{G\times \sur{G}}(Q)}=\Ind_{N_{G\times\sur{G}}(Q)}^{G\times\sur{G}}\circ\Inf_{\sur{N}_{G\times \sur{G}}(Q)}^{N_{G\times \sur{G}}(Q)}$. Here $\alpha_i$ and $\alpha'_i$ run over the irreducible characters of $\langle s\rangle$. Again by \cite[Proof of Proposition~4.1.2]{MD} for each $i$, the modules $k C_G(P)/C_{\langle s\rangle}(P)\otimes_k k_{\alpha_i}$ and $k C_G(P)/C_{\langle s\rangle}(P)\otimes_k k_{\alpha'_i}$ are projective indecomposable $k\sur{N}_{G\times \sur{G}}(Q)$-modules and $k\sur{N}_{\sur{G}\times G}(Q')$-modules respectively. Now since $k C_G(P)/C_{\langle s\rangle}(P)\otimes_k k_{\alpha_i}$ is projective indecomposable, it has the trivial group as vertex. So $\Inf^{N_{G\times \sur{G}}(Q)}_{\sur{N}_{G\times \sur{G}}(Q)}\big(k C_G(P)/C_{\langle s\rangle}(P)\otimes_k k_{\alpha_i}\big)$ has the group $Q$ as a vertex. Note that the group $Q$ is twisted diagonal. Therefore indecomposable direct summands of $\Indinf^{G\times \sur{G}}_{\sur{N}_{G\times \sur{G}}(Q)}\big(k C_G(P)/C_{\langle s\rangle}(P)\otimes_k k_{\alpha_i}\big)$ have twisted diagonal vertices, i.e. $\Indinf^{G\times \sur{G}}_{\sur{N}_{G\times \sur{G}}(Q)}\big(k C_G(P)/C_{\langle s\rangle}(P)\otimes_k k_{\alpha_i}\big)\in \FFTD(G,\sur{G})$. Similarly, we have $\Indinf^{\sur{G}\times G}_{\sur{N}_{\sur{G}\times G}(Q')}\big(k C_G(P)/C_{\langle s\rangle}(P)\otimes_k k_{\alpha'_i}\big)\in \FFTD(\sur{G},G)$. Now since $\mathcal{E}^{\Delta}(G)\neq 0$, the image of identity element $kG\in \FFTD(G,G)$ in $\mathcal{E}^{\Delta}(G)$ is non-zero. Hence we have $\sur{G}=G$, i.e. $C_{\langle s\rangle}(P)=1$. 
\end{proof}

Suppose we have $G=\Ps$ and $C_{\langle s\rangle}(P)=1$. The essential algebra $\mathcal{E}^{\Delta}(G)$ is generated by the images of the primitive idempotents 
\[
F^{G\times G}_{Q,t} = \frac{1}{\mid C_{N_{G\times G(Q)}(t)}\mid}\sum_{\substack{\varphi\in \widehat{\langle t\rangle} \\ L\leqslant Q \\ L^t=L}} \tilde{\varphi}(t^{-1})|C_L(t)|\mu((L,Q)^t)\Ind^{G\times G}_{\langle Lt\rangle} k^{\Qt}_{L,\varphi}
\]
where $Q$ is a twisted diagonal subgroup of $G\times G$. By \cite[Lemma 2.5.9]{MD}, if the image of $\Ind^{G\times G}_{\langle Lt\rangle} k^{\Qt}_{L,\varphi}$ is non-zero, then we must have that $p_1(\langle Lt\rangle)=G=p_2(\langle Lt\rangle)$. Write $t=(u,v)$. Then $p_1(\langle Lt\rangle)=\langle p_1(L)u\rangle$ and $p_2(\langle Lt\rangle)=\langle p_2(L)v\rangle$. Therefore we have $|u|=|v|=|s|$. Being a subgroup of twisted diagonal subgroup $Q$, the group $L$ itself is also twisted diagonal. Since $k_1(L)=k_2(L)=1$ and $|u|=|v|=|s|$, we have $k_1(\langle Lt\rangle)=k_2(\langle Lt\rangle)=1$. This shows that the subgroup $\langle Lt\rangle$ is twisted diagonal and $p_1(\langle Lt\rangle)=G=p_2(\langle Lt\rangle)$. Since the images of $\Ind^{G\times G}_{\langle Lt\rangle} k^{\Qt}_{L,\varphi}$ in $\mathcal{E}(G)$ with $\langle Lt\rangle$ satisfying these properties, generate the non-zero algebra $\mathcal{E}(G)$, this shows that the algebra $\mathcal{E}^{\Delta}(G)$ is also non-zero and the map $\Theta:\mathcal{E}^{\Delta}(G)\to\mathcal{E}(G)$ is surjective. Thus we have proved the following: 
\begin{prop}
The essential algebra $\mathcal{E}^{\Delta}(G)$ is non-zero if and only if there is a pair $(P,s)$ in $G$ such that $G=\Ps$ and $C_{\langle s\rangle}(P)=1$. Moreover the map $\Theta:\mathcal{E}^{\Delta}(G)\to\mathcal{E}(G)$ is surjective.
\end{prop}

Suppose we have $G=\Ps$ for some pair and $C_{\langle s\rangle}(P)=1$. We will show that the map $\Theta:\mathcal{E}^{\Delta}(G)\to\mathcal{E}(G)$ is also injective. \\
Suppose an element $\sum \overline{r_{\varphi,\alpha} kG_{\varphi,\alpha}}\in \mathcal{E}^{\Delta}(G)$ is mapped to zero by $\Theta$.  Then the element $\sum \overline{r_{\varphi,\alpha} kG_{\varphi,\alpha}}$ of $\mathcal{E}(G)$ is zero. Write
\[
\sum r_{\varphi,\alpha} kG_{\varphi,\alpha}=\sum_{|H|<|G|} t_{H,U_H,V_H} U_H\otimes_{kH} V_H
\]
for some $(kG,kH)$-bimodule $U_H$ and $(kH,kG)$-bimodule $V_H$ and some constants $t_{H,U_H,V_H}\in \FF$. Suppose the coefficient $t_{H,U_H,V_H}$ is non-zero for some group $H$. Then as in \cite{MD} we can assume that $H=\langle Rt\rangle$ for some pair $(R,t)$ and that the modules $U_H$ and $V_H$ are indecomposable. By \cite[Section 4.1]{MD} one has
\[
U_H\otimes_{kH} V_H \cong \mathrm{Indinf}_{\sur{N}_{G\times G}(\Delta(P))}^{G\times G} \bigoplus_i \big(kZ(P)\otimes k_{\lambda_i}\big)^{n_i}
\]
where $\lambda_i$ is a character of $\langle s\rangle$ and $n_i\in \NN$. Again by \cite[Section 4.1]{MD} each summand $kZ(P)\otimes k_{\lambda_i}$ is a projective indecomposable $k\sur{N}_{G\times G}(\Delta(P))$-module. This shows that if the the coefficient $t_{H,U_H,V_H}$ is non-zero, then the indecomposable direct summands of the bimodule $U_H\otimes_{kH} V_H$ have twisted diagonal vertices. Therefore the element $\sum \overline{r_{\varphi,\alpha} kG_{\varphi,\alpha}}$ is zero in $\mathcal{E}^{\Delta}(G)$. This proves that the map $\Theta:\mathcal{E}^{\Delta}(G)\to\mathcal{E}(G)$ is injective. We summarise our results as a theorem below.

\begin{theorem}\label{essentialfinal}
The essential algebra $\mathcal{E}^{\Delta}(G)$ is non-zero if and only if there is a pair $(P,s)$ in $G$ such that $G=\Ps$ and $C_{\langle s\rangle}(P)=1$. In that case, the algebra $\mathcal{E}^{\Delta}(G)$ is isomorphic to the algebra $\big(\FF[X]/\Phi_n[X]\big) \rtimes \Out(G)$ where $n$ is the order of $s$.\\
\end{theorem}

\section{$\DD$-pairs}

Let $H\leqslant G$ be a subgroup. The $(kG, kH)$-bimodule $kG$ is denoted by $\Ind_{H}^{G}$ and $(kH, kG)$-bimodule $kG$ is denoted by $\Res_{H}^{G}$. Similarly, if $N\unlhd G$ is a normal subgroup, the $(kG/N, kG)$-bimodule $kG/N$ is denoted by $\Def^{G}_{G/N}$ and $(kG, kG/N)$-bimodule $kG/N$ is denoted by $\Inf^{G}_{G/N}$. This notation is consistent with our previous use of induction, restriction, inflation and deflation symbols, in the sense that for example, if $M$ is a $kH$-module, then the induced module $\Ind_H^GM$ is isomorphic to $\Ind_H^G\otimes_{kH}M$.

We have the following lemma due to \cite{BT} and \cite{MD}.

\begin{lemma}\label{biset action}
\begin{enumerate}[(i)]
\item Let $(P,s)\in \mathcal{Q}_{G,p}$ be a pair and $H\leqslant G$ be a subgroup. Then we have 
\begin{align*}
\Res^{G}_{H} F^{G}_{P,s} =\sum_{Q,t} F^{H}_{Q,t}
\end{align*}
where $(Q,t)$ runs over a set of representatives of $H$-conjugacy classes of $G$-conjugates of $(P,s)$ contained in $H$. 
\item Let $(Q,t)\in \mathcal{Q}_{H,p}$ be a pair and $H\leqslant G$ be a subgroup. Then we have 
\begin{align*}
\Ind_{H}^{G} F^{H}_{Q,t} = |N_{G}(Q,t): N_{H}(Q,t)| F^{G}_{Q,t}.
\end{align*}
\item Let $N\unlhd G$ and $(P,s)\in \mathcal{Q}_{G/N,p}$. Then
\begin{align*}
\Inf_{G/N}^{G} F^{G/N}_{P,s} = \sum_{Q,t} F^{G}_{Q,t}
\end{align*}
where $(Q,t)$ runs over a set of representatives of $G$-conjugacy classes of pairs in $\mathcal{Q}_{G,p}$ such that $QN/N= {}^{\sur{g}}P$ and $\sur{t}={}^{g}s$ for some $\sur{g}\in G/N$. 
\item Let $N\unlhd G$ and $(P,s)\in \mathcal{Q}_{G,p}$. Then 
\begin{align*}
\Def^{G}_{G/N} F^{G}_{P,s} = m_{P,s,N}\cdot F^{G/N}_{Q,t}
\end{align*}
for some pair $(Q,t)\in \mathcal{Q}_{G/N,p}$ and a constant $m_{P,s,N}\in \FF$.\\
If $G=\Ps$ then 
\begin{align*}
\Def^{G}_{G/N} F^{G}_{P,s} = m_{P,s,N} \cdot F^{G/N}_{PN/N, \sur{s}}.
\end{align*}
\end{enumerate}
\end{lemma}

\begin{proof}
See \cite[Proposition 3.1. and Proposition 3.2.]{BT} for (i) and (ii), \cite[Proposition 3.1.3]{MD} for (iii) and \cite[Lemma 3.1.4 and Proposition 3.1.5]{MD} for (iv).
\end{proof}

\begin{lemma}
Let $N\unlhd G$ be a normal subgroup of $G$.
\begin{enumerate}[(i)]
\item We have $\Def^{G}_{G/N} \in \FFT^{\Delta}(G/N, G)$ if and only if $N$ is a $p'$-group.
\item We have $\Inf^{G}_{G/N} \in \FFT^{\Delta}(G, G/N)$ if and only if $N$ is a $p'$-group.
\end{enumerate}
\end{lemma}

\begin{proof}
(i) Let $Q\leqslant (G/N)\times G$ be a maximal vertex of an indecomposable direct summand of the $(kG/N, kG)$-bimodule $kG/N$. Equivalently $Q$ is a maximal $p$-subgroup having a fixed point on the set $G/N$. Suppose $(aN, b)\in Q$ stabilises a basis element $gN$ of $kG/N$. Then we have $(aN) gN b^{-1} = gN$ which implies that $a^{g}\cdot b^{-1}\in N$. Since the vertices of an indecomposable module are conjugate, we may assume that $g=1$. Thus, up to conjugacy, $Q$ is a Sylow $p$-subgroup of 
\begin{align*}
H=\{(aN,b) : ab^{-1}\in N\}\leqslant (G/N) \times G.
\end{align*}
Note that $k_1(Q)=k_1(H)=1$ and $k_2(Q)$ is a Sylow $p$-subgroup of $N$. Hence $Q$ is twisted diagonal if and only if $N$ is a $p'$-group. The result follows.\\
(ii) Similar. 
\end{proof}

Let $(P,s)$ be a pair and suppose $G=\langle Ps \rangle$. Then by \cite[Corollary 3.1.9]{MD} for any normal subgroup $N$ of $G$, we have the following formula for the constant $m_{P,s,N}$:
\begin{align*}
m_{P,s,N}=\frac{|s|}{|N \cap \langle s\rangle| |C_G(s)|}\sum_{\substack{Q\leqslant P \\ Q^s=Q \\ \langle Qs\rangle N=G}} |C_Q(s)| \mu\big((Q,P)^s\big).
\end{align*}

\begin{lemma}\label{defconstant}
Let $(P,s)$ be a pair and suppose $G=\Ps$. Then for any normal $p'$-subgroup $N$ of $G$ we have
\[
m_{P,s,N}=\frac{1}{|N|}\cdot 
\]
\end{lemma}
\begin{proof}
First observe that since $N$ is a $p'$-group, we have $N\leqslant C_{\langle s\rangle}(P)$. For any subgroup $Q$ of $P$ the condition $\langle Qs\rangle N=\Ps $ implies that $|Q|=|P|$ and hence $Q=P$. Therefore the formula above becomes
\begin{align*}
m_{P,s,N}=\frac{|s| |C_P(s)|}{|N| |C_G(s)|}=\frac{1}{|N|}\cdot
\end{align*}
\end{proof}

\begin{definition}
A pair $(P,s)$ is called $\DD$-pair if $\Def^{\Ps}_{\Ps/N} F^{\Ps}_{P,s} = 0$ for any nontrivial normal $p'$-subgroup $N$ of $\Ps$. 
\end{definition}

\begin{lemma}\label{DDpair}
Let $(P,s)$ be a pair. Then $(P,s)$ is a $\DD$-pair if and only if the group $\Ps$ does not have any nontrivial normal $p'$-subgroup, that is, if and only if $C_{\langle s\rangle}(P)=\nolinebreak1$. 
\end{lemma}
\begin{proof}
By Lemma \ref{defconstant}, for any normal $p'$-subgroup $N\unlhd \Ps$ we have $m_{P,s,N}=1/|N|$. Therefore $(P,s)$ is a $\DD$-pair if and only if the group $\Ps$ does not have any nontrivial normal $p'$-subgroup. The result follows.  
\end{proof}

\section{The functor $\FFT^{\Delta}$}
By \cite{Boucsimple}, the simple diagonal $p$-permutation functors are parametrized by the pairs $(G,V)$ where $G$ is a finite group and $V$ is a simple $\mathcal{E}^{\Delta}(G)$-module. Note that this implies $\mathcal{E}^{\Delta}(G)\neq 0$. 

For a simple $\mathcal{E}^{\Delta}(G)$-module $V$, we define two functors in $\FF pp_k^\Delta$ by:
\begin{align*}
L_{G,V}(H):= \FF\TD(H,G)\otimes_{\mathcal{E}^{\Delta}(G)}V
\end{align*}
and 
\begin{align*}
J_{G,V}(H):=\left\lbrace \sum_{i}\phi_i \otimes v_i \in L_{G,V} : \forall \psi \in \FFTD(G,H), \sum_i (\psi\circ\phi_i)\cdot v_i = 0 \right\rbrace,
\end{align*}
for any finite group $H$. The action of morphisms in $\FF pp_k^\Delta$ on these evaluations is given by left composition.
The functor $J_{G,V}$ is the unique maximal subfunctor of $L_{G,V}$, so the quotient  
\[
S_{G,V}:= L_{G,V}/J_{G,V}
\]
is a simple functor \cite{Boucsimple}. 

Let $\FFTD: \FF pp_k^\Delta\to \FF$-$\mathrm{Mod}$ be the functor given by 
\begin{itemize}
\item $\FFTD(G):= \FF\otimes_{\ZZ} T(G) = \FFT(G)$,
\item $\FFTD(X): \FFT(G)\to \FFT(H), M\mapsto X\otimes_{kH} M$ for any $X\in \FFTD(H,G)$. 
\end{itemize}
For any $kG$-module $X$, we denote by $\widetilde{X}$ the $(kG, kG)$-bimodule $k(G\times X)$ where the action of $kG$-$kG$ is given by
\[
a\cdot (g,x)\cdot b^{-1}:= (agb, b^{-1}x)
\]
for all $a,b,g\in G$ and $x\in X$. We have an isomorphism of $(kG,kG)$-bimodules
\[
\widetilde{X}\cong \Ind_{\delta(G)}^{G\times G^{op}}\mathrm{Iso}(\delta)(X)	
\]
where $\delta :G\to G\times G^{op}$, $g\mapsto (g,g^{-1})$. See \cite[Definition 2.5.17]{MD}. Note that the image $\delta(G)$ of $G$ in $G\times G^{op}$ is a twisted diagonal subgroup. If $X$ is an indecomposable $p$-permutation $kG$-module with a vertex $Q$, then any vertex of an indecomposable direct summand of $\widetilde{X}$ is contained in $\delta(Q)$, up to conjugation. Therefore for any $X\in \FFT(G)$ we have $\widetilde{X}\in \FFTD(G,G)$.

\begin{lemma}
Let $F$ be a subfunctor of $\FFTD$. Then for any finite group $G$, the $\FF$-vector space $F(G)$ is an ideal of the algebra $\FFTD(G)$ of $p$-permutation modules. 
\end{lemma}
\begin{proof}
Let $Y\in F(G)$ and assume $X$ is a $p$-permutation $kG$-module. By \cite[Proposition 2.5.18]{MD} we have an isomorphism $X\otimes_{k} Y\cong \widetilde{X}\otimes_{kG} Y$ of $kG$-modules. Since $\widetilde{X}\in \FFTD(G,G)$ and $F$ is a functor, we have $\widetilde{X}\otimes_{kG} Y\in F(G)$. This shows that $F(G)$ is an ideal of $\FFTD(G)$. 
\end{proof}

\begin{definition}
For any pair $(P,s)$ let $\mathbf{e}_{P,s}$ denote the subfunctor of $\FFTD$ generated by the idempotent $F^{\Ps}_{P,s}\in\FF T^\Delta\big(\langle Ps\rangle\big)$. 
\end{definition}

\begin{prop}\label{sumofsubfunctors}
Let $F$ be a subfunctor of $\FFTD$. Then we have
\[
F=\sum_{\mathbf{e}_{P,s}\leqslant F} \mathbf{e}_{P,s}.
\]
\end{prop}
\begin{proof}
Since $F$ is a subfunctor, we have 
\[
\sum_{\mathbf{e}_{P,s}\leqslant F} \mathbf{e}_{P,s} \leqslant F.
\]
Now let $G$ be a finite group, and $u=\sum_{(P,s)}\lambda_{P,s}F_{P,s}^G$, where $(P,s)$ runs in a set of representatives of $G$-conjugacy classes of $\mathcal{Q}_{G,p}$, and $\lambda_{P,s}\in\FF$. Then $F_{P,s}^G\cdot u=\lambda_{P,s}F_{P,s}^G\in F(G)$, since $F(G)$ is an ideal of $\FF T^\Delta(G)$. Hence $F_{P,s}^G\in F(G)$ if $\lambda_{P,s}\neq\nolinebreak 0$. In this case we have $\Res^G_{\Ps} F^G_{P,s}\in F\big(\Ps\big)$, which implies by Lemma~\ref{biset action} that $F^{\Ps}_{P,s}\in F\big(\Ps\big)$. This shows that $\beps\leqslant F$. By Lemma~\ref{biset action} again, $F_{P,s}^G$ is a non zero scalar multiple of $\Ind_{\langle Ps\rangle}^GF^{\Ps}_{P,s}$, so $F_{P,s}^G\in\beps(G)$, which gives finally
\[
u\in\sum_{\beps\leqslant F}\beps(G).
\]
Therefore we have  
\[
F=\sum_{\mathbf{e}_{P,s}\leqslant F} \mathbf{e}_{P,s}
\]
as desired.
\end{proof}

\begin{prop}\label{someofideals}
Let $(P_i,s_i)_{i\in I}$ be a set of pairs for an indexing set $I$. Then for any pair $(Q,t)$ we have $\mathbf{e}_{Q,t}\leqslant \sum_{i\in I} \mathbf{e}_{P_i,s_i}$ if and only if $\mathbf{e}_{Q,t}\leqslant \mathbf{e}_{P_i,s_i}$ for some $i\in I$. 
\end{prop}
\begin{proof}
If $\mathbf{e}_{Q,t}\leqslant \mathbf{e}_{P_i,s_i}$ for some $i\in I$, then we obviously have $\mathbf{e}_{Q,t}\leqslant \sum_{i\in I} \mathbf{e}_{P_i,s_i}$. Conversely assume we have $\mathbf{e}_{Q,t}\leqslant \sum_{i\in I} \mathbf{e}_{P_i,s_i}$. Then $\mathbf{e}_{Q,t}\big(\Qt\big)\leqslant \sum_{i\in I} \mathbf{e}_{P_i,s_i}\big(\Qt\big)$ and so $F^{\Qt}_{Q,t}\in \sum_{i\in I} \mathbf{e}_{P_i,s_i}\big(\Qt\big)$. Since $F^{\Qt}_{Q,t}$ is a primitive idempotent and since $\mathbf{e}_{P_i,s_i}\big(\Qt\big)$ is an ideal of $\FF T^\Delta\big(\Qt\big)$  it follows that we have $F^{\Qt}_{Q,t}\in  \mathbf{e}_{P_i,s_i}\big(\Qt\big)$ for some $i\in I$ and hence $\beqt\leqslant \mathbf{e}_{P_i,s_i}$.
\end{proof}

Let $G$ be a finite group and $(P,s)\in\mathcal{Q}_{G,p}$ be a pair such that $G=\Ps$. Let also $(Q,t)\in \mathcal{Q}^{\Delta}_{H\times G, p}$ for a finite group $H$. Suppose  that $\eta: p_1(Q)\to p_2(Q)$ is the canonical isomorphism. Up to conjugation in $H\times G$, we can assume $t=(u,s^j)$. By \cite[Section 3.2]{MD} if $p_2(\Qt)\neq G$, then the product $F^{H\times G}_{Q,t}\otimes_{kG} F^G_{P,s}$ is zero. So assume that we have $p_2(\Qt)=G$. This implies that we have $p_2(Q)=P$ and $|s^j|=|s|$. Then since $k_1(Q)=k_2(Q)=1$, this implies that we have $p_1(Q)\cong P$. Since the group $Q$ is $t$-stable, the isomorphism $\eta: p_1(Q)\to P$ commutes with conjugations by $u$ and $s^j$. Now \cite[Equation (3.3), Section 3.2]{MD} implies that as $kH$-module the product $F^{H\times G}_{Q,t}\otimes_{kG} F^G_{P,s}$ is equal to
\begin{align*}
\frac{1}{|C_{N_{H\times G}}(Q)(t)||C_G(s)|}\sum_{\substack{\varphi\in\sur{\langle t\rangle}\\ \psi\in\sur{\langle s\rangle}\\ \varphi^{|u|}\psi^{j|u|}=1}} \tilde{\varphi}(t)^{-1}\tilde{\psi}(s)^{-1} |C_Q(t)| \sum_{\substack{J\leqslant p_1(Q)\\ J^u=J}} \sigma(J)\Ind^H_{\langle Ju\rangle}(k^{\langle p_1(Q)u\rangle}_{\langle Ju\rangle, \phi}) 
\end{align*}
where $\sigma(J):= \sum_{\substack{L\leqslant P \\ L^s=L\\ \eta(J)=L}}|C_L(s)|\mu\big((L,P)^s\big)$ and $\phi(u):=\varphi(u,s^j)\psi(s)^j$.

Suppose we have $H=\langle P's'\rangle$ for a pair $(P',s')$. Then by \cite[Lemma 2.7.6]{MD} if $\tau^H_{P',s'}(F^{H\times G}_{Q,t}\otimes_{kG} F^G_{P,s})\neq 0$, then we must have $p_1(Q)=P'$ and $|u|=|s'|$. This implies in particular that we must have $P'\cong P$. Moreover again by \cite[Lemma 2.7.6]{MD} we have $\tau^H_{P',s'}\big(\Ind^H_{\langle Ju\rangle}(k^{\langle p_1(Q)u\rangle}_{\langle Ju\rangle, \phi})\big)=0$ if $J\neq P'$. Therefore if we have $P'\cong P$ then $\tau^H_{P',s'}(F^{H\times G}_{Q,t}\otimes_{kG} F^G_{P,s})$ is equal to
\[
\frac{1}{|C_{N_{H\times G}}(Q)(t)||C_G(s)|}\sum_{\substack{\varphi\in\sur{\langle t\rangle}\\ \psi\in\sur{\langle s\rangle}\\ \varphi^{|u|}\psi^{j|u|}=1}} \tilde{\varphi}(t)^{-1}\tilde{\psi}(s)^{-1} |C_Q(t)||C_P(s)|\tilde{\phi}(s').
\]

This shows that if we have $\FFTD\big(\langle P's'\rangle,\Ps\big)\otimes_{k\Ps}F^{\Ps}_{P,s}\neq 0$, then there is an isomorphism $\eta: P'\to P$ and a $p'$-element $(u,s^j)\in \langle P's'\rangle\times \Ps$ such that $\eta\circ c_u=c_{s^j}\circ\eta$ and $|u|=|s'|$, $|s^j|=|s|$. In that case, assume further that $C_{\langle s\rangle}(P)=1$. Then we have $|c_s|=|s|$ and $|c_{s^j}|=|s^j|$. Since we have $\eta\circ c_u=c_{s^j}\circ\eta$ it follows that $|c_u|=|c_{s^j}|$. Therefore we have $|s| \mid |s'|$. But then \cite[Proposition 2.3.6]{MD} implies that there is a surjective group homomorphism $\sur{\eta}:\langle P's'\rangle\to \Ps$ that induces an isomorphism of pairs $\big(P'\ker(\sur{\eta})/\ker(\sur{\eta}), s'\ker(\sur{\eta})\big)\simeq (P,s)$. Note that since $|P'|=|P|$ the order of $\ker(\sur{\eta})$ is coprime to $p$. We have the following:

\begin{lemma}\label{equiv2}
Let $(P,s)$ be a pair with $C_{\langle s\rangle}(P)=1$ and set $G:=\Ps$. Let $H$ be a finite group. The following statements are equivalent:
\begin{enumerate}[(i)]
\item $\FFTD(H,G)\otimes_{kG} F^G_{P,s} \neq 0$.
\item There exists a pair $(P',s')$ contained in $H$ such that the pair $(P,s)$ is isomorphic to a $p'$-quotient of the pair $(P',s')$, that is, there exists a normal $p'$-subgroup $K$ of $\langle P's'\rangle$ such that $(P,s)\simeq (P'K/K,s'K)$.
\end{enumerate}
\end{lemma}
\begin{proof}
(i) $\Rightarrow$ (ii) Suppose we have $\FFTD(H,G)\otimes_{kG} F^G_{P,s} \neq 0$. Then there exists a pair $(P',s')$ in $H$ such that
\[
F^H_{P',s'}\in \FFTD(H,G)\otimes_{kG} F^G_{P,s}.
\]
Via the restriction map this implies that we have
\[
F^{\langle P's'\rangle}_{P',s'}\in \FFTD\big(\langle P's'\rangle,G\big)\otimes_{kG} F^G_{P,s}.
\]
Therefore by the argument above we have an isomorphism $(P'K/K,s'K)\simeq (P,s)$ of pairs where $K$ is a normal $p'$-subgroup of $\langle P's'\rangle$. \\
(ii) $\Rightarrow$ (i) Suppose $\Phi: (P'K/K,s'K)\to (P,s)$ is an isomorphism of pairs where $K$ is a normal $p'$-subgroup of $\langle P's'\rangle$. Then we have 
\[
\Ind^H_{\langle P's'\rangle}\Inf^{\langle P's'\rangle}_{\langle P's'\rangle/K}\Isom(\Phi)F^G_{P,s}\neq 0.
\]
This shows (i).
\end{proof}

\begin{prop}
Let $(P,s)$ be a pair. The following are equivalent:
\begin{enumerate}[(i)]
\item $(P,s)$ is a $\DD$-pair.
\item For any finite group $H$ with $|H| < |\Ps|$, we have $\beps(H)=\{0\}$.
\item If $H$ is a finite group with $\beps(H)\neq \{0\}$, then the pair $(P,s)$ is isomorphic to a $p'$-quotient of a pair $(P',s')$ contained in $H$. 
\item For any nontrivial normal $p'$-subgroup $N$ of $\Ps$, we have $\Def^{\Ps}_{\Ps/N} F^{\Ps}_{P,s}=0$.
\item The group $\Ps$ does not have any nontrivial normal $p'$-subgroup.
\item We have $C_{\langle s\rangle}(P)=1$.
\end{enumerate}
\end{prop}
\begin{proof}
(vi)$\Leftrightarrow$(v)$\Leftrightarrow$(i) : This follows from Lemma \ref{DDpair}.\\
(iv)$\Leftrightarrow$ (i): This follows from the definition of $\DD$-pairs.\\
(i)$\Rightarrow$ (iii): Since $(P,s)$ is a $\DD$-pair, we have $C_{\langle s\rangle}(P)=1$. So (iii) follows from Lemma \ref{equiv2}.\\
(iii)$\Rightarrow$ (ii): Assume that (iii) holds and $\beps(H)\neq 0$ where $H$ is a finite group with $\lvert H\rvert <\lvert \Ps\rvert$. Then by the assumption, we have $\lvert H\rvert \geq \lvert \langle P's'\rangle\rvert\geq \lvert \Ps\rvert$. Contradiction. \\
(ii)$\Rightarrow$ (iv): Clear. 
\end{proof}

\begin{prop}\label{inclusions}
Let $(P,s)$ and $(Q,t)$ be two pairs.
\begin{enumerate}[(i)]
\item If the pair $(Q,t)$ is isomorphic to a $p'$-quotient of the pair $(P,s)$, then we have $\beps = \beqt$. 
\item If $(Q,t)$ is a $\DD$-pair, and if $\beps\leqslant \beqt$, then $(Q,t)$ is isomorphic to a $p'$-quotient of $(P,s)$. 
\end{enumerate}
\end{prop}
\begin{proof}
(i) Assume we have an isomorphism $\phi: (PK/K, sK)\to (Q,t)$ of pairs for some normal $p'$-subgroup $K$ of $\Ps$. Then we have 
\[
F^{\Ps}_{P,s}\otimes_{k} \Inf^{\Ps}_{\Ps/K}\Isom(\phi^{-1})F^{\Qt}_{Q,t}\neq 0.
\]
Therefore we have $F^{\Ps}_{P,s}\in \beqt\big(\Ps\big)$ which implies that $\beps\leqslant \beqt$. 

Now we also have 
\[
F^{\Qt}_{Q,t}\otimes_k \Isom(\phi)\Def^{\Ps}_{\Ps/K} F^{\Ps}_{P,s}\neq 0
\] 
which implies that $F^{\Qt}_{Q,t}\in \beps\big(\Qt\big)$. Therefore we have $\beqt\leqslant\beps$ and so $\beqt=\beps$ as desired. \\
(ii) Since $\beps\leqslant \beqt$, we have $F^{\Ps}_{P,s}\in \beqt\big(\Ps\big)$. Since $(Q,t)$ is a $\DD$-pair, by the proof of Lemma \ref{equiv2}, there exists a normal $p'$-subgroup $K$ of $\Ps$ such that $(Q,t)\simeq (PK/K,sK)$.  
\end{proof}

\begin{prop}\label{evaluation}
Let $F$ be a nonzero subfunctor of $\FFTD$. If $H$ is a minimal group of $F$, then $H=\Qt$ for some $\DD$-pair $(Q,t)$. Moreover we have 
\[
F(H)\leqslant \bigoplus_{\substack{(Q',t'), \DD-pair \\ \langle Q't'\rangle=H}} \FF F^H_{Q',t'}
\]
and $\beqt\leqslant F$. 

In particular, if $F=\beqt$ for some $\DD$-pair $(Q,t)$, then we have 
\[
\beqt\big(\Qt\big) = \bigoplus_{\substack{(Q',t')\simeq (Q,t)\\ \langle Q't'\rangle =\Qt}} \FF F^H_{Q',t'}.
\]
\end{prop}
\begin{proof}
Let $F$ be a nonzero subfunctor of $\FFTD$ and assume $H$ is a minimal group of~$F$. Since $F(H)\neq 0$, there exists a pair $(Q,t)\in \mathcal{Q}_{H,p}$ such that $F^H_{Q,t}\in F(H)$. This implies, via the restriction map, that we have $F^{\Qt}_{Q,t}\in F\big(\Qt\big)$. Since $H$ is a minimal group, this implies that we have $H=\Qt$. Now if $N$ is a normal $p'$-subgroup of $\Qt$, then $\Def^{\Qt}_{\Qt/N} F^{\Qt}_{Q,t}=\frac{1}{|N|} F^{\Qt/N}_{QN/N,tN}\neq 0$. Again since $H$ is a minimal group this means that $N$ is trivial and hence the pair $(Q,t)$ is a $\DD$-pair. It follows moreover that
\[
F(H)\leqslant \bigoplus_{\substack{(Q',t'), \DD-pair \\ \langle Q't'\rangle=H}} \FF F^H_{Q',t'}.
\]
For the last part, consider the functor $\beqt$ for some $\DD$-pair $(Q,t)$. If $F^{\Qt}_{Q',t'}\in \beqt\big(\Qt\big)$ for some $\DD$-pair $(Q',t')$, then by the second part of Proposition \ref{inclusions}, the pair $(Q,t)$ is isomorphic to a $p'$-quotient of the pair $(Q',t')$. But the pair $(Q',t')$ is contained in $\Qt$. Thus we have $(Q',t')\simeq (Q,t)$.\\ Conversely, if the pairs $(Q',t')$ and $(Q,t)$ are isomorphic via a map $\phi$, then we have $F^{\Qt}_{Q',t'}=\Isom(\phi)F^{\Qt}_{Q,t}$. Therefore we have
\[
\beqt\big(\Qt\big) = \bigoplus_{\substack{(Q',t')\simeq (Q,t)\\ \langle Q't'\rangle =\Qt}} \FF F^H_{Q',t'}.
\]
\end{proof}

Let $(P,s)$ be a pair and $N$ a normal $p'$-subgroup of $\Ps$. Then the pair $(PN/N,sN)$ is a $p'$-quotient of the pair $(P,s)$ and so by Proposition \ref{inclusions} we have $\beps=\mathbf{e}_{PN/N,sN}$.

\begin{prop}\label{uniqueminimal}
Let $(P,s)$ be a pair. Then the group $\Ps/C_{\langle s\rangle}(P)$ is the unique, up to isomorphism, minimal group of the functor $\beps$. Moreover there is a unique isomorphism class of $\DD$-pairs $(P',s')$ such that $\langle P's'\rangle\cong \Ps/C_{\langle s\rangle}(P)$ and we have $\mathbf{e}_{P',s'}=\beps$. Furthermore we have $(P',s')\simeq \big(PC_{\langle s\rangle}(P)/C_{\langle s\rangle}(P),sC_{\langle s\rangle}(P)\big)$.
\end{prop}
\begin{proof}
Let $(P',s')$ be a $\DD$-pair such that $\langle P's'\rangle$ is a minimal group of the functor $\beps$. By Proposition \ref{evaluation}, we have $\mathbf{e}_{P',s'}\leqslant \beps$. Let $N:=C_{\langle s\rangle}(P)$. Then the pair $(PN/N,sN)$ is a $\DD$-pair, and we have $\beps=\mathbf{e}_{PN/N,sN}$. Since $(PN/N,sN)$ is a $\DD$-pair, by Proposition \ref{inclusions} there exists a normal $p'$-subgroup $K$ of $\langle P's'\rangle$ such that $(P'K/K, s'K)\simeq (PN/N,sN)$. This means that the idempotent $F^{\langle P's'\rangle /K}_{P'K/K,s'K}$ is in the evaluation at $\langle P's'\rangle /K$ of the functor $\mathbf{e}_{PN/N,sN}=\beps$.  Since the group $\langle P's'\rangle$ is a minimal group of $\beps$ it follows that we must have $K=1$. Thus we have $(P',s')\simeq (PN/N,sN)$. Therefore we have $\mathbf{e}_{P',s'}=\textbf{e}_{PN/N,sN}=\beps$. 

Now we will show the uniqueness of the isomorphism class of the minimal groups of $\beps$. Let $H$ be a minimal group of $\beps$. It suffices to show that $H$ is isomorphic to $\langle P's'\rangle$. By Proposition \ref{evaluation} the group $H$ is of the form $H=\Qt$ for some $\DD$-pair $(Q,t)$. By the first part of the proof we have $\beqt = \beps = \mathbf{e}_{P',s'}$. Since both $(Q,t)$ and $(P,s)$ are $\DD$-pairs, the equality $\beqt = \mathbf{e}_{P',s'}$ implies that $(Q,t)$ is isomorphic to a $p'$-quotient of $(P,s)$, and vice versa. Therefore we have $(Q,t)\simeq (P',s')$ which implies that $H=\Qt\cong \langle P's'\rangle$ as desired. 
\end{proof}

For any pair $(P,s)$ we denote by $(\tilde{P},\tilde{s})$ a representative of the isomorphism class of the pair $(PC_{\langle s\rangle}(P)/C_{\langle s\rangle}(P),sC_{\langle s\rangle}(P))$.

\begin{theorem}\label{long}
Let $(P,s)$ be a pair.
\begin{enumerate}[(i)]
\item If $(Q,t)$ is isomorphic to a $p'$-quotient of $(P,s)$ and if $(Q,t)$ is a $\DD$-pair, then $(Q,t)$ is isomorphic to the pair $(\tilde{P},\tilde{s})$. In particular, for any normal $p'$-subgroup $N\unlhd \Ps$, we have $(PN/N, sN)\simeq (\tilde{P},\tilde{s})$ if and only if $(PN/N,sN)$ is a $\DD$-pair.
\item Let $N\unlhd \Ps$ be a normal $p'$-subgroup. Then the pair  $(\tilde{P},\tilde{s})$ is isomorphic to a $p'$-quotient of $(PN/N,sN)$ and we have $(\tilde{P},\tilde{s})\simeq (\widetilde{PN/N},\widetilde{sN})$.
\end{enumerate}
\end{theorem}
\begin{proof}
(i) Since the pair $(Q,t)$ is isomorphic to a $p'$-quotient of the pair $(P,s)$, by Proposition \ref{inclusions}, we have $\mathbf{e}_{\tilde{P},\tilde{s}}=\beps \leqslant \beqt$. Since $(Q,t)$ is a $\DD$-pair, again by Proposition \ref{inclusions}, the pair $(Q,t)$ is isomorphic to a $p'$-quotient of $(\tilde{P},\tilde{s})$. But since the pair $(\tilde{P},\tilde{s})$ is a $\DD$-pair, it follows that the pair $(Q,t)$ is isomorphic to the pair $(\tilde{P},\tilde{s})$.

(ii) Since the constant $m_{P,s,N}$ is non-zero, we have $F^{\Ps/N}_{PN/N,sN}\in \beps\big(\Ps/N\big)=\textbf{e}_{\tilde{P},\tilde{s}}\big(\Ps/N\big)$. Therefore we have $\mathbf{e}_{PN/N,sN}\leqslant \textbf{e}_{\tilde{P},\tilde{s}}$ and since $(\tilde{P},\tilde{s})$ is a $\DD$-pair, by Proposition \ref{inclusions}, $(\tilde{P},\tilde{s})$ is isomorphic to a $p'$-quotient of $(PN/N,sN)$. Again since the pair $(\tilde{P},\tilde{s})$ is a $\DD$-pair, by part (i), it is isomorphic to the pair $(\widetilde{PN/N},\widetilde{sN})$. 
\end{proof}

Let $[\DD\textit{-pair}]$ denote a set of isomorphism classes of $\DD$-pairs. Then the subfunctor lattice of the functor $\FFTD$ is isomorphic to the lattice of subsets of the set $[\DD\textit{-pair}]$ ordered by inclusion.

\begin{theorem}\label{subfunctorlattice}
Let $\mathcal{S}$ be the lattice of subfunctors of $\FFTD$ ordered by inclusion of subfunctors. Let $\mathcal{T}$ be the lattice of subsets of $[\DD\textit{-pair}]$ ordered by inclusion of subsets. Then the map
\begin{align*}
\Theta : \mathcal{S} \to \mathcal{T}
\end{align*}
that sends a subfunctor $F$ to the set $\{(P,s)\in [\DD\textit{-pair}]: \beps \leqslant F\}$, is an isomorphism of lattices with inverse 
\begin{align*}
\Psi:\mathcal{T}\to \mathcal{S}
\end{align*}
that sends a subset $A$ to the functor $\sum_{(P,s)\in A} \beps$.
\end{theorem}
\begin{proof}
We need to show that the maps $\Theta$ and $\Psi$ are inverse of each other. Let $F\in \mathcal{S}$ be a subfunctor. By Proposition \ref{sumofsubfunctors} we have
\[
F=\sum_{\substack{(P,s)\in \Gamma \\ \beps\leqslant F}} \beps
\]
where $\Gamma$ is a set of representatives of the isomorphism classes of pairs. But for any pair $(P,s)$ we have $\beps=\mathbf{e}_{\tilde{P},\tilde{s}}$ and $(\tilde{P},\tilde{s})$ is a $\DD$-pair. Therefore we have 
\[
F=\sum_{\substack{(P,s)\in [\DD\textit{-pair}] \\ \beps\leqslant F}}\beps.
\]
This shows that $\Psi(\Theta(F))=F$. 

Now let $A\in \mathcal{T}$ be a subset and let $(Q,t)\in \Theta(\Psi(A))$ be a $\DD$-pair. Then we have $\beqt \leqslant \sum_{(P,s)\in A}\beps$ and so by Proposition \ref{someofideals} this implies that we have $\beqt\leqslant\beps$ for some $(P,s)\in A$. Since both $(P,s)$ and $(Q,t)$ are $\DD$-pairs, it follows that $(P,s)\simeq (Q,t)$ and hence $(Q,t)\in A$. This shows that $\Theta(\Psi(A))\subseteq A$. The inclusion $A\subseteq \Theta(\Psi(A))$ is trivial. Therefore we have $\Theta(\Psi(A))=A$. 
\end{proof}

The following corollary follows immediately from Theorem \ref{subfunctorlattice}.

\begin{corollary}
We have $\FFTD=\bigoplus_{(P,s)\in [\DD\textit{-pair}]} \beps$.
\end{corollary}

The first statement of Proposition \ref{evaluation} can also be made stronger.
\begin{corollary}
Let $F$ be a nonzero subfunctor of $\FFTD$. If $H$ is a minimal group of~$F$, then $H=\Qt$ for some $\DD$-pair $(Q,t)$ and we have 
\begin{align*}
F(H)=\bigoplus_{\substack{(Q',t')\simeq (Q,t)\\ \langle Q't'\rangle=\Qt}} \FF F^H_{Q',t'}.
\end{align*}
\end{corollary}
\begin{proof}
Since $H$ is a minimal group of $F$, by Proposition \ref{evaluation} it follows that $H=\Qt$ for some $\DD$-pair with the property that $\beqt\leqslant F$. By Theorem \ref{subfunctorlattice} we have 
\[
F=\sum_{\substack{(Q,t)\in [\DD\textit{-pair}] \\ \beqt\leqslant F}}\beqt.
\]
Therefore by Proposition \ref{evaluation} again we have
\begin{align*}
F(H)=\beqt(H)= \bigoplus_{\substack{(Q',t')\simeq (Q,t)\\ \langle Q't'\rangle=\Qt}} \FF F^H_{Q',t'}
\end{align*}
as desired.
\end{proof}

\begin{corollary}
Let $(P,s)$ be a $\DD$-pair. Then the subfunctor $\beps$ of $\FFTD$ is isomorphic to the simple functor $S_{\Ps,W_{P,s}}$ where $W_{P,s}=\oplus_{\substack{(Q,t)\simeq (P,s)\\ \Qt=\Ps}} \FF F^{\Ps}_{P,s}$.
\end{corollary}
\begin{proof}
By Theorem \ref{subfunctorlattice} the lattice of subfunctors of $\beps$ is isomorphic to the lattice of subsets of the set $\Theta(\beps)=\{(Q,t)\in[\DD\textit{-pair}]: \beqt\leqslant\beps\}=\{(P,s)\}$. Therefore the subfunctor $\beps$ is simple. By Proposition \ref{uniqueminimal} the group $\Ps$ is a minimal group of the functor $\beps$. By Proposition \ref{evaluation} we have $\beps\big(\Ps\big) =W_{P,s}$. Moreover, by \cite[Theorem 4.2.5]{MD}, the module $W_{P,s}$ is a simple module for the essential algebra $\mathcal{E}^{\Delta}\big(\Ps\big)$. Thus we have $\beps\simeq S_{\Ps,W_{P,s}}$ as desired. 
\end{proof}

\begin{prop}\label{simplequotient}
If $F\leqslant F'$ are subfunctors of $\FFTD$ such that $F'/F$ is simple, then there exists a unique $\DD$-pair $(P,s)\in [\DD\textit{-pair}]$ such that $\beps \leqslant F'$ and $\beps \nleqslant F$. In particular, we have $\beps + F=F'$, $\beps\cap F=\{0\}$, and $F'/F\simeq S_{\Ps,W_{P,s}}$
\end{prop}
\begin{proof}
The existence of a pair $(P,s)$ with the property that $\beps\leqslant F'$ and $\beps \nleqslant F$ is clear. Suppose $(P',s')$ is another pair with these properties. Since $F'/F$ is simple, we have
\[
(P',s')\in \Theta(F) \cup \{(P,s)\}.
\]
Thus $(P',s')\simeq (P,s)$ as $(P',s')\notin \Theta(F)$. Now since $\beps \nleqslant F$ and $F'/F$ is simple, we have $\beps + F=F'$. Thus the quotient $\beps/(\beps\cap F)\simeq F'/F$ is simple and so $\beps\cap F=\{0\}$. Therefore we have $F'/F\simeq S_{\Ps,W_{P,s}}$. 
\end{proof}

\begin{prop}
Let $F\leqslant F'$ be subfunctors of $\FFTD$ such that $F'/F$ is simple. Let $H$ (respectively $H'$) be a finite group and $W$ (respectively $W'$) be a simple $\mathcal{E}^{\Delta}(H)$-mod (respectively  $\mathcal{E}^{\Delta}(H')$-mod) such that $S_{H,W}\simeq S_{H',W'}\simeq F'/F$. Then $H\cong H'$. Moreover $W\cong W'$, after identification of $H$ and $H'$ via the previous isomorphism.
\end{prop}
\begin{proof}
By Proposition \ref{simplequotient} there exists a unique $\DD$-pair $(P,s)$ such that $F'/F\simeq S_{\Ps,W_{P,s}}$. Therefore it suffices to prove that $H\cong \Ps$. Since $(F'/F)(H)\neq 0$ there exists a pair $(Q,t)$ contained in $H$ such that $F^H_{Q,t}\in F'(H)\setminus F(H)$. Since $H$ is a minimal group of $F'/F$, it follows that $H=\Qt$ and $(Q,t)$ is a $\DD$-pair. Moreover we have $\beqt\leqslant F'$ and $\beqt\nleqslant F$. But the pair $(P,s)$ is the unique $\DD$-pair with these properties. Therefore we have $(Q,t)\simeq (P,s)$. Thus $H\cong \Ps$ as desired. The last assertion follows from the fact that $S_{H,W}(H)\cong W$.
\end{proof}

\begin{prop}\label{generators}
Let $(P,s)$ be a pair. Then for any finite group $H$, the $\FF$-vector space $\beps(H)$ is the subspace of $\FFT(H)$ generated by the set of primitive idempotents $F^H_{Q,t}$ where $(Q,t)$ runs over a set of conjugacy classes of pairs in $H$ with the property that $(P,s)$ is isomorphic to a $p'$-quotient of $(Q,t)$.
\end{prop}
\begin{proof}
Since the pair $(\tilde{P},\tilde{s})$ is isomorphic to a $p'$-quotient of the pair $(P,s)$ and since $\beps =\textbf{e}_{\tilde{P},\tilde{s}}$, we may assume that the pair $(P,s)$ is a $\DD$-pair. Since $\beps(H)$ is an ideal of $\FFT(H)$, it has a $\FF$-basis consisting of a set of primitive idempotents $F^H_{Q,t}$. If $F^H_{Q,t}\in \beps(H)$, then $F^{\Qt}_{Q,t}\in \beps\big(\Qt\big)$ and so $\beqt\leqslant \beps$. Since $(P,s)$ is a $\DD$-pair, by Proposition \ref{inclusions}, it is isomorphic to a $p'$-quotient of the pair $(Q,t)$. Conversely, if $(P,s)$ is isomorphic to a $p'$-quotient of the pair $(Q,t)$, then again by Proposition \ref{inclusions}, we have $\beqt\leqslant \beps$. So we have $F^{\Qt}_{Q,t}\in \beps\big(\Qt\big)$ and hence $F^H_{Q,t}\in \beps(H)$. The result follows. 
\end{proof}

\begin{theorem}\label{dimension}
Let $(P,s)$ be a $\DD$-pair. Then for any finite group $H$, the $\FF$-dimension of $S_{\Ps,W_{P,s}}(H)$ is equal to the number of conjugacy classes of pairs $(Q,t)$ in $H$ such that $(\tilde{Q},\tilde{t})\simeq (P,s)$.
\end{theorem}
\begin{proof}
By Proposition \ref{generators} $\beps(H)$ is generated by the idempotents $F^H_{Q,t}$ where $(Q,t)$ is a pair in $H$ with the property that the pair $(\tilde{P},\tilde{s})\simeq (P,s)$ is isomorphic to a $p'$-quotient of the pair $(Q,t)$. Since $(P,s)$ is a $\DD$-pair, Theorem \ref{long} implies that $(\tilde{Q},\tilde{t})\simeq (P,s)$. The result follows.
\end{proof}

\begin{corollary}
Let $H$ be a finite group. The $\FF$-dimension of $S_{1,\FF}(H)$ is equal to the number of isomorphism classes of simple $kH$-modules. 
\end{corollary}
\begin{proof}
By Theorem \ref{dimension}, $\mathrm{dim}_{\FF}S_{1,\FF}(H)$ is equal to the number of conjugacy classes of pairs $(Q,t)$ in $H$ such that $(\tilde{Q},\tilde{t})\simeq (1,1)$. Suppose $(Q,t)$ is a pair with $(\tilde{Q},\tilde{t})\simeq (1,1)$. Then we have $\tilde{Q}=1$ and $\tilde{t}=1$. So there exists a normal $p'$-subgroup $N$ of $\Qt$ such that $(QN/N,tN)\simeq (1,1)$. Since $|Q|$ and $|N|$ are coprime, this implies that $Q=1$. We also have $t\in N$. But then $N\unlhd \langle t\rangle$ implies that $N=\langle t\rangle$. Therefore the number of conjugacy classes of pairs $(Q,t)$ in $H$ such that $(\tilde{Q},\tilde{t})\simeq (1,1)$ is equal to the number of conjugacy classes of $p'$-elements in $H$. The result follows.
\end{proof}

\begin{theorem}\label{projectivefunctor}
The functor $S_{1,\FF}$ is isomorphic to the functor that sends a finite group $H$ to the subspace $\FF K_0(kH)$ of $\FFT^\Delta(H)$ generated by the projective indecomposable $kH$-modules.
\end{theorem}
\begin{proof}
Let $H$ be a finite group. We have
\[
S_{1,\FF}(H)=(\FFTD(H,1)\otimes_{\FF} \FF )/J_{1,\FF}(H)\cong \FFTD(H,1)/J_{1,\FF}(H)
\]
where $J_{1,\FF}(H)= \{\phi\in \FFTD(H,1): \forall \psi \in \FFTD(1,H), (\psi\circ \phi)\cdot 1=0\}$. Now $\FFTD(H,1)$ is isomorphic to the subspace $\FF K_0(kH)$ of $\FFT(H)$ generated by the isomorphism classes of projective indecomposable $kH$-modules. Similarly any $W\in \FFTD(1,H)$ can be identified with $W^*\in \FF K_0(kH)$. As in \cite{MDP} we have the following:\\
For any $p$-permutation $kH$-modules $V$ and $W$ we have
\[
(W^*\otimes_{kH} V)\cdot 1=\mathrm{dim}_{k}(W^*\otimes_{kH} V)=\mathrm{dim}_{k}(\Hom_{kH}(W,V)).
\]
Therefore $J_{1,\FF}(H)$ is the right kernel of the bilinear form
\[
<-,->: \FF K_0(kH)\to \FF
\]
defined as $<W,V>:=\mathrm{dim}_k(\Hom_{kH}(W,V))$.
But the matrix that represents this bilinear form is the Cartan matrix of $kH$. Since the Cartan matrix of a group algebra is non-degenerate, it follows that $J_{1,\FF}(H)=0$. Therefore we have 
\[
S_{1,\FF}(H)=\FFTD(H,1)\otimes_{\FF} \FF\cong \FFTD(H,1)\cong \FF K_0(kH).
\]
Note that both of these isomorphisms are functorial in $H$. The result follows. 
\end{proof}





\begin{thebibliography}{00}

\bibitem{BoltjeXu} Boltje R., Xu B., ``On $p$-permutation equivalences: between Rickard equivalences and isotypies", \textit{Trans. Amer. Math. Soc.} 360(10): 5067-5087, 2008.

\bibitem{Boucsimple} Bouc S., ``Foncteurs d'ensembles munis d'une double action", \textit{J. Algebra} 183 (3): 664-736, 1996.

\bibitem{BT} Bouc S., Th{\'e}venaz J., ``The primitive idempotents of the $p$-permutation ring'', \textit{J. Algebra} 323(10): 2905-2915, 2010. 

\bibitem{Broue} Brou{\'e} M., ``On Scott modules and $p$-permutation modules: An approach through the Brauer morphism", \textit{Proc. Amer. Math. Soc.} 93(3): 401-408, 1985.

\bibitem{MD} Ducellier M., ``Foncteurs de $p$-permutation" \textit{PhD Diss.}, Universit\'e de Picardie Jules Verne, 2015. https://hal.archives-ouvertes.fr/tel-02146174

\bibitem{MDP} Ducellier M., ``A study of a simple $p$-permutation functor'', \textit{J. Algebra} 447: 367-382, 2016.

\bibitem{Philipp} Perepelitsky P., ``$p$-permutation equivalences between blocks of finite groups" ProQuest LLC, Ann Arbor, MI, 2014. Thesis (Ph.D.) University of California, Santa Cruz.

\bibitem{Rickard} Rickard J., ``Splendid equivalences: derived categories and permutation modules", \textit{Proc. London Math Soc.} (3) 72(2): 331-358, 1996.
\end{thebibliography}


\end{document}